\documentclass[10pt, english]{article}

\usepackage[margin= 2.3 cm]{geometry}
\usepackage{authblk}
\usepackage{amsthm}
\usepackage{amsmath}
\usepackage{amssymb}
\usepackage{setspace}
\usepackage{mathtools}
\usepackage{verbatim}
\usepackage{csquotes}
\usepackage{graphicx}
\usepackage[hidelinks]{hyperref}
\usepackage{thm-restate}
\usepackage{cleveref}
\usepackage{graphicx}
\usepackage{appendix}
\usepackage[inline]{enumitem}
\usepackage{framed}
\usepackage{subcaption}

\usepackage{caption}

\usepackage{floatrow}
\usepackage[T1]{fontenc}

\usepackage{tikz}
\usepackage{mathdots}
\usepackage{xcolor}
\usepackage{diagbox}
\usepackage{colortbl}
\usepackage[absolute,overlay]{textpos}

\usetikzlibrary{calc}
\usetikzlibrary{decorations.pathreplacing}
\usetikzlibrary{positioning,patterns}
\usetikzlibrary{arrows,shapes,positioning}
\usetikzlibrary{decorations.markings}

\tikzstyle{edge}=[very thick]
\definecolor{bostonuniversityred}{rgb}{0.8, 0.0, 0.0}
\definecolor{arsenic}{rgb}{0.23, 0.27, 0.29}
\tikzstyle{diredge}=[postaction={decorate,decoration={markings,
		mark=at position .65 with {\arrow[scale = 1]{stealth};}}}]

\theoremstyle{plain}

\newtheorem*{thm*}{Theorem}
\newtheorem{thm}{Theorem}
\Crefname{thm}{Theorem}{Theorems}

\newtheorem*{lem*}{Lemma}

\Crefname{lem}{Lemma}{Lemmas}

\newtheorem*{claim*}{Claim}

\Crefname{claim}{Claim}{Claims}
\Crefname{claim}{Claim}{Claims}

\newtheorem{prop}[thm]{Proposition}
\Crefname{prop}{Proposition}{Propositions}

\Crefname{cor}{Corollary}{Corollaries}

\Crefname{conj}{Conjecture}{Conjectures}

\Crefname{qn}{Question}{Questions}

\Crefname{obs}{Observation}{Observations}

\newtheorem{ex}[thm]{Example}
\Crefname{ex}{Example}{Examples}

\theoremstyle{definition}

\Crefname{prob}{Problem}{Problems}

\Crefname{defn}{Definition}{Definitions}

\newtheorem*{defn*}{Definition}

\theoremstyle{remark}

\renewenvironment{proof}[1][]{\begin{trivlist}
\item[\hspace{\labelsep}{\bf\noindent Proof#1.\/}] }{\qed\end{trivlist}}

\newcommand{\R}{\mathbb{R}}

\expandafter\def\expandafter\normalsize\expandafter{%
    \normalsize
    \setlength\abovedisplayskip{8pt}
    \setlength\belowdisplayskip{8pt}
    \setlength\abovedisplayshortskip{4pt}
    \setlength\belowdisplayshortskip{4pt}
}

\usepackage[square,sort,comma,numbers]{natbib}
 \setlist[itemize]{leftmargin=*}

\DeclareFontFamily{OT1}{pzc}{}
\DeclareFontShape{OT1}{pzc}{m}{it}{<-> s * [1.10] pzcmi7t}{}
\DeclareMathAlphabet{\mathpzc}{OT1}{pzc}{m}{it}

\DeclareMathOperator{\spn}{span}

\title{Geometric graphs with exponential \\ chromatic number and arbitrary girth}

\author{}
 \author[1]{Matija Bucić \thanks{Research supported in part by an NSF Grant DMS--2349013.}}
\author[2]{James Davies}

 \affil[1]{Princeton University, USA}
 \affil[2]{University of Cambridge, UK.}

\date{}

\begin{document}

\maketitle

\begin{abstract}
    In 1975 Erd\H{o}s initiated the study of the following very natural question. What can be said about the chromatic number of unit distance graphs in $\mathbb{R}^2$ that have large girth? Over the years this question and its natural extension to  $\mathbb{R}^d$ attracted considerable attention with the high-dimensional variant reiterated recently by Alon and Kupavskii.
    
    We prove that there exist unit distance graphs in $\mathbb{R}^d$ with chromatic number at least $(1.074 + o(1))^d$ that have arbitrarily large girth. This improves upon a series of results due to Kupavskii; Sagdeev; and Sagdeev and Raigorodskii and gives the first bound in which the base of the exponent does not tend to one with the girth.
    In addition, our construction can be made explicit which allows us to answer in a strong form a question of Kupavskii.
    
    Our arguments show graphs of large chromatic number and high girth exist in a number of other geometric settings, including diameter graphs and orthogonality graphs.
\end{abstract}

\section{Introduction}
Chromatic number and girth are perhaps two of the most fundamental graph parameters, the study of which has attracted a vast amount of attention over the years. The \emph{chromatic number} of a graph $G$ (denoted by $\chi(G)$) is the minimum number of colours needed to colour its vertices so that no two adjacent vertices receive the same colour (such a colouring is referred to as \emph{proper}),  
and the \emph{girth} of a graph is equal to the length of a shortest cycle in the graph. One of the most classical directions of study of these two parameters explores the relation between them. In particular, the question of whether one can obtain some amount of control of the chromatic number in terms of the girth has a very long and illustrious history. Indeed, at first glance, it is highly non-trivial to construct graphs with very high chromatic number, where this can not be witnessed locally. The local behaviour is captured quite well by the assumption of having large girth since this is equivalent to saying that the graph needs to locally look like a forest and any forest has chromatic number at most two. 

That graphs with arbitrarily large girth and chromatic number do exist is a classical result of Erd\H{o}s \cite{erdos59-high-girth-high-chromatic-graph} from 1958, which features in essentially any course on the probabilistic method in combinatorics. Numerous alternative constructions with a variety of interesting further features were found over the years. Let us highlight the first explicit construction due to Lov\'asz from 1966 \cite{lovasz-explicit}, the extension to hypergraphs by Erd\H{o}s and Hajnal \cite{erdos-hajnal66-high-girth-high-chromatic-hypergraph} from the same year and the construction of Ramanujan graphs by Lubotzky, Phillips and Sarnak \cite{ramanujan} from 1988, which besides high girth and chromatic number have a variety of other useful features associated with expander graphs and have had a tremendous impact in a variety of areas over the years. We point the reader to surveys by Ne\u{s}et\u{r}il \cite{nesetril-survey} and Raigorodskii and Shabanov \cite{raigorodskii-shabanov-survey} collecting just some of the very rich history surrounding this question. 

The study of the chromatic number of a variety of geometric graphs is equally as classical a topic. For example, the Hadwiger-Nelson problem \cite{hadwiger-nelson}, dating back to 1944, asks for the maximum chromatic number of a unit distance graph in $\mathbb{R}^d$ while the classical Borsuk Problem \cite{borsuk}, dating back to 1933, asks for the maximum chromatic number of a diameter graph in $\mathbb{R}^d$. 
Here a graph is said to be a \emph{unit distance} graph in $\mathbb{R}^d$ if we can assign its vertices to distinct points in $\mathbb{R}^d$ in such a way that any pair of adjacent vertices are assigned to points at a unit distance\footnote{We note that here and throughout the paper we do not insist that this assignment is faithful, namely we do not require that any pair of points at distance one actually correspond to an edge in the graph, we discuss this distinction further in \Cref{sec:con}.}. A graph is said to be a \emph{diameter graph} in $\mathbb{R}^d$ if we can assign its vertices to distinct points in $\mathbb{R}^d$ in such a way that any pair of adjacent vertices are assigned to a pair of points at the maximum distance. We note that in both settings one may assume graphs involved are finite via the classical de-Bruijn-Erd\H{o}s Compactness Theorem \cite{de-Bruijn-Erdos}. 

In the unit distance setting, in their remarkable 1981 paper essentially introducing the algebraic method in combinatorics, Frankl and Wilson \cite{frankl-wilson} proved a conjecture of Erd\H{o}s that there exist unit distance graphs in $\mathbb{R}^d$ with chromatic number exponential in $d$, improving upon a sequence of progressively better polynomial bounds (see e.g.\ the survey \cite{hadwiger-nelson-survey} for the full history). The Frankl-Wilson construction has been slightly improved since by Raigorodskii \cite{raigorodskii-bound} who holds the current record of $\chi(G) \ge (1.239+o(1))^d$ for a unit distance graph in $\mathbb{R}^d$. The best known upper bound $(3+o(1))^d$ is due to Larman and Rogers \cite{larman-rogers} (see also \cite{prosanov-hadwiger-nelson} for a more recent simpler proof due to Prosanov). 

In the diameter graph setting, no bound beyond $\chi(G) = d+1$, obtained by the regular simplex, was known. It was widely believed and is referred to as the Borsuk Conjecture that $\chi(G) \le d+1$ for any diameter graph in $\mathbb{R}^d$, based in part on numerous results establishing this for special configurations of points. This was disproved in a stunning fashion by Kahn and Kalai in \cite{kahn-kalai} who constructed diameter graphs with chromatic number growing exponentially in $\sqrt{d}$. There have been several improvements in the base, as well as a considerable effort in trying to prove any non-trivial bounds for small $d$, but constructing a diameter graph with exponential chromatic number remains a very interesting open problem. The current record is $\chi(G) \ge (1.225+o(1))^{\sqrt{d}}$ for $G$ a diameter graph in $\mathbb{R}^d$, due to Raigorodskii \cite{borsuk-best}. The best known upper bound here is $(3/2)^{(1+o(1))d/2}$ due to Schramm \cite{schramm}, see also \cite{bourgain-lindenstrauss} for a different proof due to Bourgain and Lindenstrauss. See the survey \cite{borsuk-survey} for a more complete history and plenty of related problems. For example, in \cite{borsuk-alternative} the authors construct counterexamples to Borsuk Conjecture in any large enough dimension which are in a certain sense very different compared to the Kahn-Kalai construction. 

A very natural next question, with the goal of understanding further what kind of substructures one can find in geometric graphs, is to ask whether one can find geometric graphs of high chromatic number and high girth. This question is quite classical in its own right and traces its origins to a problem posed by Erd\H{o}s in 1975 \cite{erdos1975unsolved} asking whether a unit distance graph in $\mathbb{R}^2$ with chromatic number $4$ and girth at least $4$ exists. Soon after he reiterated the problem saying it is likely that such a graph does not exist. Wormald \cite{wormald} disproved this in 1978 by exhibiting such a graph on 6448 vertices. A construction on 23 vertices was found in 1996 by Hochberg and O'Donnell \cite{hochberg}. Erd\H{o}s in the same 1975 paper also asked the more general question of whether unit distance graphs in $\mathbb{R}^2$ with chromatic number $4$ and arbitrarily large girth exist. This was only settled in 1999 by O'Donnell \cite{odonell,odonnell2}, see also a recent survey of Graham \cite{graham}. Natural higher dimensional extensions were first considered by Raigorodskii \cite{raigorodskii07} and Rubanov  \cite{rubanov} in 2007. For more details and a plethora of related results see e.g.\ the survey by Raigorodskii \cite{raigorodski-survey}.

Kupavskii \cite{kupavskii2012distance} in 2012 proved that for each positive integer $g$ there exists some $c_g>1$ such that there are unit distance graphs in $\mathbb{R}^d$ with girth at least $g$ and chromatic number at least $(c_g+o(1))^d$. Alon and Kupavskii \cite{alon-kupavskii} reiterate the question of what can be said about the chromatic number of such graphs.
Sagdeev \cite{sagdeev2017lower} gave the first explicit estimate for $c_g$, namely that $c_g\ge 1 + 1/{g^{4+o(1)}}$. The dependency of $c_g$ on $g$ was subsequently improved to $c_g\ge 1 + 1/{g^{2+o(1)}}$ in \cite{sagdeev2018improved} and by a further constant factor (implicit in $1/{g^{2+o(1)}}$) by Sagdeev and Raigorodskii in \cite{sagdeev2019frankl}. We prove a first universal exponential lower bound, without any dependency on $g$.

\begin{thm}\label{main}
    There exist unit distance graphs in $\mathbb{R}^d$ with chromatic number at least $(1.074 + o(1))^d$ that have arbitrarily large girth.
\end{thm}

Our bound even improves the best-known bounds for small values of $g$, e.g.\ in perhaps the most classical case of lower bounding the chromatic number of unit distance triangle-free graphs in $\mathbb{R}^d$ the previously best bound was $(1.058+o(1))^d$ \cite{demekhin2012distance}.

Quite a bit of focus over the years was afforded to finding explicit constructions of high girth and high chromatic number graphs, ever since the first such construction was found by Lov\'asz in 1966 \cite{lovasz-explicit}, in large part motivated by the idea that such graphs should exhibit somewhat random-like behaviour so could be useful in a variety of derandomisation arguments.
Explicit constructions of graphs with large girth and chromatic number are now known in some geometric settings, including as box and line intersection graphs in $\mathbb{R}^3$ \cite{Davies2021Box}, as disjointedness graphs of short polygonal chains in the plane \cite{PACH202429}, and as tangency graphs of circles in the plane \cite{davies2021solution}.
Our construction behind Theorem \ref{main} is actually more flexible than previous, weaker ones in the unit distance graph setting and can be made explicit. This answers in quite a strong form a question of Kupavskii who asked to prove explicitly his result that unit distance graphs in $\mathbb{R}^d$ with girth at least $g$ and chromatic number at least $(c_g+o(1))^d$ exist for some $c_g>1$.

Our ideas can also be used to find graphs with high chromatic number and high girth in other geometric settings. The question of constructing high chromatic number and high girth diameter graphs in $\mathbb{R}^d$ for large $d$ has been raised recently by Prosanov \cite{prosanov2019counterexamples}, motivated by a result of Kupavskii and Polyanski \cite{schur}. We show that indeed strong counterexamples to Borsuk's Conjecture can be found with arbitrarily large girth. 
\begin{thm}\label{thm:diam}
    There exist diameter graphs in $\mathbb{R}^d$ with chromatic number at least $(1.107 + o(1))^{\sqrt{d}}$ that have arbitrarily large girth.
\end{thm}

Another setting, underlying and connecting the previous two results, is that of orthogonality graphs. Here a graph is said to be an \emph{orthogonality} graph in $S \subseteq \mathbb{R}^d$ if we can assign its vertices to distinct vectors in $S$ in such a way that any pair of adjacent vertices are assigned to orthogonal vectors. In this paper, we will mostly work with orthogonality graphs over spheres, here we denote by $\mathbb{S}^{d-1}\subseteq \mathbb{R}^d$ the $d-1$ dimensional unit sphere consisting of all vectors of unit length in $\mathbb{R}^d$. 

\begin{thm}\label{thm:orth}
    There exist orthogonality graphs in $\mathbb{S}^{d-1}$ with chromatic number at least $(1.074 + o(1))^d$ that have arbitrarily large girth.
\end{thm}

In \Cref{sec:self} we shall give a self-contained proof of \Cref{thm:orth} with a somewhat weaker bound for the base of the exponent. We also prove this weaker bound for \Cref{main} and \Cref{thm:diam} and more generally show that \Cref{thm:orth} implies \Cref{main} and \Cref{thm:diam}.
Then in \Cref{sec:black} we give some more general black-box arguments that can be combined with state-of-the-art bounds to obtain \Cref{thm:orth} (and thus \Cref{main} and \Cref{thm:diam}). In this section, we make our construction explicit by replacing a probabilistic sub-sampling step from \Cref{sec:self} with an explicit construction for the desired sub-sampling (see \Cref{lem:explicit blowup}).
Lastly, in \Cref{sec:con} we give some concluding remarks and discuss further open problems.

\section{A self-contained argument}\label{sec:self}
In this section, we will give a self-contained proof of  \Cref{thm:orth} with a slightly weaker bound, in the interest of keeping the arguments as simple as possible. The goal is to showcase the wonderful interplay between algebraic and probabilistic ideas happening in the background. Since many of these ideas are by now standard in the area, experts may wish to skip to the following section where we extract the main ideas in the hope of making them more easily applicable in other settings and use them combined with existing state-of-the-art results to prove the bounds stated in the previous section.

Let $g$ be an arbitrary integer and let $p>2$ be a prime. We let $d=8p$ and will construct a unit distance graph in $\mathbb{R}^d$ with girth at least $g$ and chromatic number at least $\left(1.067+o(1)\right)^d$.\footnote{We note that for $d$ not of this form, one could take $p$ to be the largest prime smaller than $d/8$ and use the known results on density of primes \cite{dusart} to conclude the loss involved in using the construction in $8p$ dimensions only impacts the $o(1)$ term.}

Let us now describe the construction. For a positive integer $m$, we denote the set $\{1, \ldots, m\}$ by $[m]$. In our construction, $m$ shall be some large integer to be chosen later. The vertex set of our graph will be 
$$V:=\left\{\left(\textbf{v}\cdot \cos \frac{2\pi j}{m},\textbf{v}\cdot \sin \frac{2\pi j}{m}\right) \mid \textbf{v} \in \{\pm1\}^{d/2}, \textbf{v}_{d/2}=1, \sum_{i=1}^{d/2} \textbf{v}_i = 0, j \in [m] \right\} \subseteq \R^d.$$
The edge set will be obtained by randomly subsampling orthogonal pairs of vectors in $V$. Let us however first introduce some notation. We let $V':= \{\textbf{v} \mid \textbf{v} \in \{\pm1\}^{d/2}, \textbf{v}_{d/2}=1, \sum_{i=1}^{d/2} \textbf{v}_i = 0\}$ and for $\textbf{v} \in V', j \in [m]$ we let $\textbf{v}(j)=\left(\textbf{v}\cdot \cos \frac{2\pi j}{m},\textbf{v}\cdot \sin \frac{2\pi j}{m}\right)$. So that $V= \{\textbf{v}(j)\mid \textbf{v}\in V', j \in [m] \} \subseteq \frac{d}{2}\cdot \mathbb{S}^{d-1}$. One can think of $V$ as consisting of $m$ spaced out rotations of each $\textbf{v} \in V'$ in the plane defined by $(\textbf{v},\textbf{0})$ and $(\textbf{0},\textbf{v})$ in $\R^d$.

The key property of $V$ is that if $\textbf{v}$ is orthogonal to $\textbf{u}$ for $\textbf{v},\textbf{u} \in V'$ then $\textbf{v}(i)$ is orthogonal to $\textbf{u}(j)$ for any $i,j \in [m]$. So if one denotes by $G'$ the graph with vertex set $V'$ and edge set being orthogonal pairs in $V'$, then the blow-up graph $G'^{m}$ (which is the graph obtained from $G'$ by replacing each vertex with an independent set of size $m$ and each edge with a $K_{m,m}$ subgraph between the two corresponding independent sets of size $m$) is an orthogonality graph in $\frac{d}{2}\cdot \mathbb{S}^{d-1}$ with vertex set $V$ (so also in $\mathbb{S}^{d-1}$ after scaling by a factor of $2/d$). Armed with this observation our plan is to show that $\chi(G')$ is exponential in $d$ and that thanks to that if we subsample the edges of $G'^{m}$ we can preserve the same chromatic number while removing all cycles of length shorter than $g$.

That the first property holds follows from a slightly modified argument of Nilli \cite{nilli}, which is, in turn, a slight modification of the one used by Frankl and Wilson \cite{frankl-wilson} when first establishing exponential growth of the chromatic number of the unit distance graph of $\R^d$. 

\begin{thm}
$\chi(G') \ge \left({27}/{16}+o(1)\right)^p \ge \left(1.067+o(1)\right)^d.$
\end{thm}
\begin{proof}
Let $S \subseteq V'$ contain no orthogonal pair of vectors. We first show that this implies $|S| \le 2\binom{4p}{p-1}$. 
To see this let us define polynomials $f_{\textbf{v}} : \{\pm 1\}^{4p} \to \mathbb{F}_p$ for each $\textbf{v} \in V'$ as $$f_{\textbf{v}}(\textbf{x}):=\prod_{i=1}^{p-1}(\langle \textbf{x},\textbf{v}\rangle-i) \bmod p.$$ Note that $f_{\textbf{v}}(\textbf{u}) \not\equiv 0 \pmod p$ for $\textbf{v},\textbf{u} \in V'$ if and only if $\langle \textbf{v}, \textbf{u}\rangle \equiv 0 \pmod p$. Note further that $\langle \textbf{v}, \textbf{u}\rangle=4(t-p)$ where $t$ is the number of $i$ such that $v_i=u_i=1$. Since $\textbf{v},\textbf{u} \in V'$ we have $v_{4p}=u_{4p}$ so $t \ge 1$ and both $\textbf{v},\textbf{u}$ have exactly $2p$ ones. So, $f_{\textbf{v}}(\textbf{u}) \not\equiv 0 \pmod p$ can only happen if $\textbf{v}=\textbf{u}$ (if $t=2p$) or $ \textbf{v} \perp \textbf{u}$ (if $t=p$).
Now define for every $f_{\textbf{v}}(\textbf{x})$ a polynomial $\overline{f}_{\textbf{v}}(\textbf{x})$ in the same space by replacing any occurrence of the term $x_i^j$ by $x_i^{j \bmod 2}$ which are equal for $x_i \in \{\pm 1\}$ so $f_{\textbf{v}}(\textbf{u})=\overline{f}_{\textbf{v}}(\textbf{u})$ for any $\textbf{u} \in V'$. This implies the polynomials $\overline{f}_{\textbf{v}}$ for $\textbf{v} \in S$ are linearly independent over $\mathbb{F}_p$ since $\overline{f}_{\textbf{v}}(\textbf{v}) \not\equiv 0 \bmod p$ while $\overline{f}_{\textbf{v}}(\textbf{u}) \equiv 0 \bmod p$ for any distinct $\textbf{v},\textbf{u}\in S$ (since $\textbf{v}$ is not orthogonal to $\textbf{u}$ by definition of $S$). Note that the polynomials $\overline{f}_{\textbf{v}}$ are spanned by the set of monomials $\{x_{i_1}x_{i_2}\cdots x_{i_k} \mid 1\le i_1 <i_2 <\ldots<i_k\le 4p;\: k \le p-1\}$. So we found $|S|$ linearly independent elements belonging to a subspace spanned by at most $\binom{4p}{0}+\binom{4p}{1}+\ldots+\binom{4p}{p-1} \le 2\binom{4p}{p-1}$ elements, and hence $|S| \le 2\binom{4p}{p-1}$.

Finally, as in any proper colouring, every colour class spans no edges, this implies that we need at least 
$$\frac{|V'|}{2\binom{4p}{p-1}} = \frac{\binom{4p}{2p}}{4\binom{4p}{p-1}} = \left({27}/{16}+o(1)\right)^p \ge \left(1.067+o(1)\right)^d$$ 
colours to properly colour $G'$, as claimed. We note that the term $27/16$ arises from Stirling's approximation which we used to estimate the ratio of the two binomials. 
\end{proof}

We now turn to the second part of the argument, namely increasing the girth. Let $k=\chi(G')$ and let $m=2^{8pg}\cdot (4k)^{4g}\cdot (k-1)^2$. As we already argued $G'^m$ is an orthogonality graph in $\mathbb{S}^{d-1}$.  Let $n=m\cdot 2^{4p} \ge |V(G'^m)|.$ Let $G_0 \sim G'^m(q)$ be the random graph obtained by including every edge of $G'^m$ in $G_0$ with probability $q:=1/n^{1-1/(2g)}$, independently from all other edges.  The expected number of cycles of length $\ell$ in $G_0$ is at most $\binom{n}{\ell} \cdot \frac{\ell!}{2\ell} \cdot q^\ell \le (nq)^\ell/2=n^{\ell/(2g)}/2$. So, the expected number of cycles of length at most $g$ is at most $\sum_{\ell=3}^{g} n^{\ell/(2g)}/2 \le n^{1/2}$ (using that $n^{1/(2g)} \ge 2$ which holds for our choice of $n$).
By Markov's inequality, we can conclude that the probability there are more than $2\sqrt{n}$ cycles in $G_0$ is less than $1/2$.

On the other hand, if we write $m':=\frac{m}{k-1}$ we want to show that with probability less than $1/2$ there exist adjacent $\textbf{v},\textbf{u} \in V(G')$ and a pair of subsets $W,U$ of $m'$ vertices of $G'^m$, corresponding to $\textbf{v}$ and $\textbf{u}$ respectively with at most $2\sqrt{n}$ edges between $W$ and $U$. To see that note that there are at most $e(G') \cdot \binom{m}{m'}^2$ choices for $W$ and $U$ and the chance there are at most $2\sqrt{n}$ edges between them is at most $\binom{m'^2}{2\sqrt{n}}(1-q)^{m'^2-2\sqrt{n}}$. So the desired statement holds by a union bound as 
$$e(G') \cdot \binom{m}{m'}^2 \cdot \binom{m'^2}{2\sqrt{n}}(1-q)^{m'^2-2\sqrt{n}}\le 2^{8p} \cdot 2^{2m} \cdot 2^{4\sqrt{n}\log_2 m'}\cdot 2^{-qm'^2+1}<\frac12,$$
where the final inequality follows since $qm'^2 \ge m^{1+1/(2g)}/(2^{4p}k^2) > 4m\ge 8p+2m+4\sqrt{n}\log_2 m +2.$

So there exists an outcome in which graph $G_0$ has at most $2\sqrt{n}$ cycles of length up to $g$ \emph{and} for any adjacent $\textbf{v},\textbf{u} \in V(G')$ and for any pair of subsets $W,U$ of $m'$ vertices of $G'^m$, corresponding to $\textbf{v}$ and $\textbf{u}$ respectively, there are more than $2\sqrt{n}$ edges between $W$ and $U$. We obtain our graph $G$ by removing an arbitrary edge from every cycle of length up to $g$ in $G_0$ to obtain a graph with girth at least $g$ and with the property that there is an edge between any pair of subsets $W,U$ of $m'$ vertices of $G'^m$, corresponding to $\textbf{v}$ and $\textbf{u}$ respectively.

Suppose that $\chi(G) < \chi(G')=k$. This means there exists a colouring $c: V(G) \to [k-1]$ such that there are no adjacent vertices with the same colour. Now colour every vertex $\textbf{v}$ of $G'$ by the majority colour according to $c$ appearing on the $m$ vertices corresponding to $\textbf{v}$ in $V(G'^m)=V(G).$ This produces a $(k-1)$-colouring of $G'$ which can not be proper, as $\chi(G')=k$, and hence there must exist adjacent $\textbf{v},\textbf{u} \in V(G')$ given the same colour. By definition of the colouring, there exist two subsets $W,U$ of $\frac{m}{k-1}=m'$ vertices of $G'^m$, corresponding to $\textbf{v}$ and $\textbf{u}$ respectively, all given the same colour in $c$. But there exists an edge between these $W$ and $U$ by our construction, which contradicts the definition of $c$.

This completes the proof of \Cref{thm:orth} with a somewhat weaker base of the exponential (being 1.067).
Next, we obtain \Cref{main} and \Cref{thm:diam} with weaker base of the exponentials.
The same arguments also show that \Cref{thm:orth} implies \Cref{main} and \Cref{thm:diam} with bounds as stated.

We start with the classical, folklore observation that any orthogonality graph on a sphere $\mathbb{S}^{d-1}$ is also a unit distance graph in $\mathbb{R}^d$. Note that all the vertices in our set $W$ have the same norm $$||\textbf{v}(j)||^2=||\textbf{v}||^2\left(\cos^2 \frac{2\pi j}m+\sin^2 \frac{2\pi j}m\right)=||\textbf{v}||^2=4p.$$ This implies that any orthogonal pair of vectors $\textbf{v}(i),\textbf{v}(j)$ in $V$ satisfy 
\begin{equation}\label{eq:diameter}
    ||\textbf{v}(i)-\textbf{v}(j)||^2=||\textbf{v}(i)||^2+||\textbf{v}(j)||^2-2\langle\textbf{v}(i),\textbf{v}(j)\rangle = 8p.
\end{equation} So after scaling all the vectors by a factor of $\sqrt{8p}$ we obtain our graph $G$ is also a unit distance graph, proving \Cref{main} with a smaller base of the exponential. We note that, in general, if we are given an orthogonality graph in $\mathbb{R}^d$ we can scale all of its vectors to have the same norm without impacting the orthogonality relations making it an orthogonality graph on $\mathbb{S}^{d-1}$ so long as no vectors scale into the same point. This is the reason we insist on working over $\mathbb{S}^{d-1}$ rather than $\mathbb{R}^d$.

We next showcase the beautiful argument due to Kahn and Kalai which allows us to convert an orthogonality graph into a diameter graph. The starting point will again be \eqref{eq:diameter} which tells us that many points in our construction come at the same distance. The issue, however, is that this distance might not be the maximum one since $\langle\textbf{v}(i),\textbf{v}(j)\rangle$ may be negative. The brilliant insight of Kahn and Kalai here is that by using a simple tensorisation trick one can replace every vector in our collection with one in $d^2$ dimensions in such a way that the inner product of any pair of our new vectors is non-negative. Indeed, we define the vertex set of our new graph as $V \otimes V:=\{\textbf{v} \otimes \textbf{v}\mid \textbf{v} \in V\}\subseteq \R^{d^2}$ where $\textbf{v} \otimes \textbf{w}$ is defined as a $d^2$ dimensional vector whose coordinates are pairwise products of coordinates of $\textbf{v}$ and $\textbf{w}$. The key property of the tensor product hinted at above is 
$$\langle \textbf{v} \otimes \textbf{v},\textbf{w} \otimes \textbf{w}\rangle =\sum_{i,j}\textbf{v}_i\textbf{v}_j\textbf{w}_i\textbf{w}_j=\sum_i\textbf{v}_i\textbf{w}_i\cdot \sum_j\textbf{v}_j\textbf{w}_j=\langle\textbf{v},\textbf{w}\rangle^2\ge 0.$$
Now \eqref{eq:diameter} becomes $||\textbf{v} \otimes \textbf{v}-\textbf{w} \otimes \textbf{w}||^2=2\cdot (4p)^2-2\cdot \langle\textbf{v},\textbf{w}\rangle^2,$ so the maximum distance between $\textbf{v} \otimes \textbf{v},\textbf{w} \otimes \textbf{w} \in V \otimes V$ is attained if and only if $\textbf{v}$ is orthogonal to $\textbf{w}$. This shows that our graph $G$ is also a diameter graph in $\mathbb{R}^{d^2}$ and proves \Cref{thm:diam} with a smaller base of the exponential. We note that one can save a factor of about a half in the dimension in this argument by noticing that $\textbf{v} \otimes \textbf{v}$ all belong to a subspace of dimension at most $d+\binom{d}{2}$ as the product of $i$-th and $j$-th coordinate appears twice. The above argument shows that any orthogonality graph in $\mathbb{S}^{d-1}$ is also a diameter graph in $\mathbb{R}^{d^2}$. Similarly as for the unit distance graphs, the same is true for orthogonality graphs over $\mathbb{R}^d$ provided none of the vectors when divided by their norm scale into the same point.

\section{Black-box arguments}\label{sec:black}
In this section, we extract the main ideas used in the previous one. We state them in a rather general form to make them more readily available for use in the future. Finally, we combine them together with the state-of-the-art lower bounds on the chromatic number without any girth restrictions to prove our main results.

The first ingredient is the state-of-the-art construction of orthogonality graphs with large chromatic number due to Raigorodskii \cite{borsuk-best}. We note that the construction is similar in spirit to the one we presented in the previous section with the key additional idea being to use vectors with coordinates in $\{-1,0,1\}$ instead of just $\{-1,1\}$.

\begin{thm}[Raigorodskii \cite{borsuk-best}]\label{thm:exponential-chi-orthogonality}
    There exist orthogonality graphs in $\mathbb{S}^{d-1}$ with chromatic number at least \linebreak $(2/\sqrt{3}+o(1))^d \ge (1.1547+o(1))^d$.
\end{thm}

The second ingredient is the fact that by doubling the number of dimensions one can embed an arbitrary blow-up of an orthogonality graph. Recall that for a graph $G$ and a positive integer $m$, the \emph{blow-up graph} $G^{m}$ is the graph obtained from $G$ by replacing each vertex with an independent set of size $m$ and each edge with a $K_{m,m}$ subgraph between the two corresponding independent sets of size $m$.

\begin{prop}\label{thm-double-dimension}
    For any $m \in \mathbb{N}$ and orthogonality graph $G$ in $\mathbb{S}^{d-1}$, $G^{m}$ is an orthogonality graph in $\mathbb{S}^{2d-1}$.
\end{prop}
\begin{proof}
    Let $G$ be an orthogonality graph in $\mathbb{S}^{d-1} \subseteq \mathbb{R}^d$ with representation $V\subseteq \mathbb{S}^{d-1}$.
    For each $\textbf{v}\in V$, let $\textbf{v}'=(\textbf{v},\textbf{0}) \in \mathbb{S}^{d-1} \times \mathbb{R}^d \subseteq \mathbb{R}^{2d}$, and $\textbf{v}''=(\textbf{0},\textbf{v}) \in \mathbb{R}^d \times \mathbb{S}^{d-1} \subseteq \mathbb{R}^{2d}$. Then, for each orthogonal pair $\textbf{u},\textbf{v}\in V$, we have that $\spn_{\mathbb{R}}\{\textbf{u}',\textbf{u}''\}$ and $\spn_{\mathbb{R}}\{\textbf{v}',\textbf{v}''\}$ are orthogonal planes in $\mathbb{R}^{2d}$.
    For each $\textbf{v}\in V$, choose distinct points $\textbf{v}_1,\ldots, \textbf{v}_m$ of unit length in $\spn_{\mathbb{R}}\{\textbf{v}',\textbf{v}''\}$ so that the points are also distinct between different $\textbf{v}\in V$. 
    Now observe that for every orthogonal pair $\textbf{u},\textbf{v}\in V$, since $\spn_{\mathbb{R}}\{\textbf{u}',\textbf{u}''\}$ and  $\spn_{\mathbb{R}}\{\textbf{v}',\textbf{v}''\}$  are orthogonal planes, we have for every $1\le i,j\le m$ that $\textbf{u}_i$ is orthogonal to $\textbf{v}_j$.
    Therefore, $G^{m}$ is an orthogonality graph in $\mathbb{S}^{2d-1}$.
\end{proof}

The probabilistic part of the argument can be captured by the following theorem see e.g.\ \cite{zhu1996uniquely}. An alternative probabilistic proof, based on the Lov\'asz Local Lemma, of a somewhat stronger result can be found in \cite{mohar2023subgraphs}.

\begin{thm}[Zhu \cite{zhu1996uniquely}]\label{lem:blowup}
    For every graph $G$ and positive integer $g$, there exists a positive integer $m$ such that $G^{m}$ contains a subgraph $G'$ with girth at least $g$ and $\chi(G')=\chi(G)$.
\end{thm}

In the following theorem, we prove a slight strengthening which comes with an explicit proof. Loosely speaking our result says that given any graph $G$ we can find a subgraph $G'$ of some blow-up graph of $G$ with large girth and with the property that any proper $k$-colouring of $G'$ gives rise to a proper $k$-coloring of $G$ (with the property that a colour of $v \in V(G)$ appears as the colour of some vertex in the part corresponding to $v$ in the blow-up). 

Our argument is a modification of one of Ne{\v{s}}et{\v{r}}il and R{\"o}dl \cite{nevsetvril1979short} that they used to construct graphs with large girth and chromatic number, which itself is a modification of one of Tutte \cite{descartes1947,descartes1954solution} (using the pseudonym Blanche Descartes) who constructed triangle-free graphs with large chromatic number.
As part of our argument, we will need the existence of (explicit) hypergraphs with high girth and chromatic number.
A \emph{cycle} of length $\ell \ge 2$ in a hypergraph $H$ is a tuple $(F_1, \ldots , F_\ell)$ of distinct hyperedges such that
there exist distinct elements $v_1, \ldots , v_\ell \in V(H)$ with $v_i \in F_i \cap F_{i+1}$ for $1\le i \le \ell-1$, and $v_\ell \in  F_\ell \cap F_1$.
The \emph{girth} of a hypergraph is equal to the length of a shortest cycle.
The \emph{chromatic number} of a hypergraph is defined as the minimum number of independent sets (sets not containing any edge fully) needed to partition its vertex set. The existence of hypergraphs of arbitrary uniformity with high girth and chromatic number follows from an immediate strengthening of the classical probabilistic argument of Erd\H{o}s \cite{erdos59-high-girth-high-chromatic-graph} and is due to Erd\H{o}s and Hajnal \cite{erdos-hajnal66-high-girth-high-chromatic-hypergraph}. Multiple explicit constructions have been found since see e.g.\ \cite{hypergraphs-high-chromatic-and-girth1,hypergraphs-high-chromatic-and-girth2,local-lemma-paper,hypergraphs-high-chromatic-and-girth3,lovasz-explicit,hypergraphs-high-chromatic-and-girth6,promel1988sparse,PV90}.
Given a vertex $u$ of a graph $G$, we let $N_G(u)$ denote the neighbourhood of $u$ in $G$. A \emph{homomorphism} from a graph $G'$ to a graph $G$ is a function from $V(G') \to V(G)$ such that if $vu \in E(G')$ then $h(v)h(u) \in E(G)$. Note that there exists a homomorphism from $G'$ to $G$ if and only if $G'$ is a subgraph of a blow-up of $G$.

\begin{thm}\label{lem:explicit blowup}
    For every graph $G$ and pair of positive integers $g,k$
    there exists a graph $G'$ with girth at least $g$ and a surjective homomorpishm $h$ from $G'$ to $G$ such that for any proper $k$-colouring $c'$ of $G'$, there exists a proper $k$-colouring $c$ of $G$ such that $c(v)\in c'(h^{-1}(v))$ for all $v\in V(G)$.
    In particular, $\chi(G')=\chi(G)$.
\end{thm}

\begin{proof}
    If $|V(G)|\le 2$ the theorem is immediate since we may just take $G'=G$ and let $h$ be the identity. So, with this observation serving as a basis, we shall argue by induction on the number of vertices (while keeping $g$ and $k$ fixed).
    Let $u$ be a vertex of $G$, then we may assume that the theorem holds for $G-u$. Now, let $G_u$ be a graph with girth at least $g$ and $h_u$ a surjective homomorphism  from $G_u$ to $G - u$ such that for any proper $k$-colouring $c_u$ of $G_u$, there exists a proper $k$-colouring $c$ of $G-u$ such that $c(v)\in c_u(h_u^{-1}(v))$ for all $v\in V(G-u)$.

    Let $n=|h_u^{-1}(N_G(u))|$, and let $H$ be an auxiliary $n$-uniform hypergraph with chromatic number at least $k+1$ and girth at least $g$. 
    For each $F\in E(H)$, let $G_F$, $h_F$ be copies of $G_u,h_u$, where the $G_F$'s are taken to be vertex disjoint.
    
    We construct our desired $G'$ by taking a vertex disjoint union of all $G_F$ together with a disjoint independent set $U$ whose vertices we identify with $V(H)$. In addition, for each $F\in E(H)$, we add an arbitrary matching between the vertices of $U$ corresponding to those of $F$ and $h_F^{-1}(N_{G}(u))\subseteq V(G_F)$. 
    We define our desired $h:V(G')\to V(G)$ by letting $h(v)=h_F(v)$ for all $v\in V(G_F)$ and $h(v)=u$ for all $v\in U$. Since $h_F$ is a surjective homomorphism to $G-u$, so is $h$ to $G$.
    For an illustration of this construction, see Figure 1.

    \begin{figure}[h]\label{figure}
\caption{An illustration of the construction for $G'$ when $G$ is the Moser spindle.
The labels of the vertices of $G_u$ gives the homomorphism from $G_u$ to $G-u$.
For simplicity, we only show how one copy $G_F$ of $G_u$ attaches to vertices $F$ of $U$ in $G'$. In reality, we place one such copy for every edge of $H$.}
\centering

\includegraphics[width=0.75\textwidth]{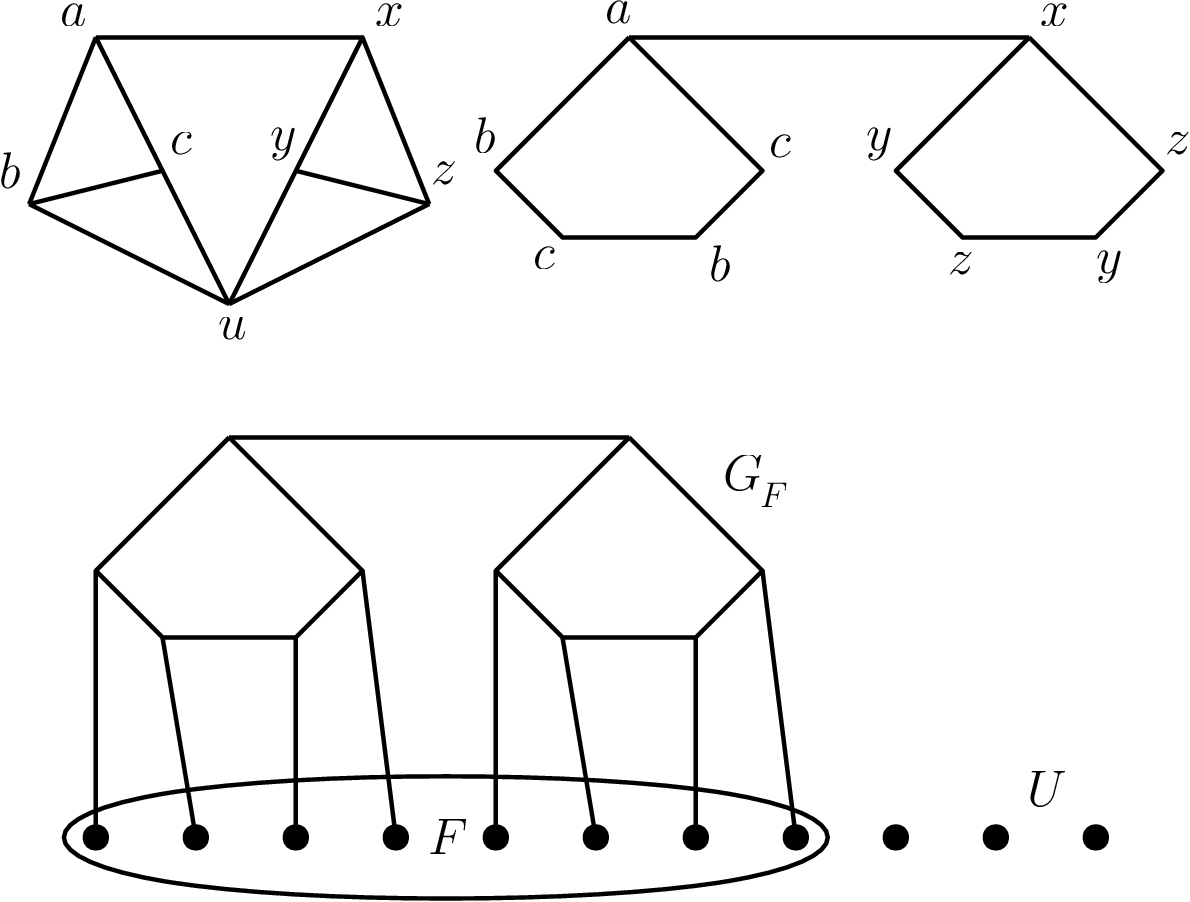}
\end{figure}

    Consider a cycle $C$ in $G'$. If $C$ is contained in some $G_F$, then clearly $C$ has length at least $g$.
    Otherwise, if $C$ is not contained in some $G_F$, then $C$ must contain some vertices $u_1,\ldots, u_r$ of $U$ which we label according to the order in which they appear on $C$ (choosing $u_1$ and direction arbitrarily). Observe that for each $1\le i \le r$, the internal vertices of the subpath of $C$ between $u_i$ and $u_{i+1}$ (taking $u_{r+1}=u_1$) are contained in a unique $G_{F_i}$, where $F_i\in E(H)$, with $u_i,u_{i+1}\in F_i$. This follows since $u_i$ and $u_{i+1}$ can not be adjacent on the cycle (since they belong to $U$ which is independent) and since there are no edges between distinct $G_F$. Furthermore, for each $i$ we have that $F_i\not= F_{i+1}$ (again taking $F_{r+1}=F_1$), since the edge from $u_{i+1}$ to $G_{F_i}$ is unique and can be traversed only once by $C$.
    Therefore, $H$ has a cycle consisting of (not necessarily all of) the edges $F_1,\ldots, F_r$.
    Since $H$ has girth at least $g$, it follows that $r\ge g$.
    Therefore, in either case, $C$ has length at least $g$.
    So, $G'$ has girth at least $g$.

    Consider some proper $k$-colouring $c'$ of $G'$. Restriction of $c'$ to $U$ gives a $k$-coloring of $H$, which has chromatic number at least $k+1$, so some edge of $H$ must be monochromatic in this coloring. Namely, there is some colour $a$ and some $F\in E(H)$ such that $c'(F)=\{a\}$.
    Therefore, since $c'$ was proper, and by our construction there is a matching between $F \subseteq U$ and $h_F^{-1}(N_{G}(u))$, we conclude  $c'(h_F^{-1}(N_{G}(u)))\subseteq \{1,\ldots, k\} \backslash \{a\}$. So, by applying the inductive assumption to the restriction of $c'$ to $G_F$ we obtain a proper $k$-colouring $c$ of $G-u$ such that $c(v)\in c'(h^{-1}(v))$ for all $v\in V(G-u)$. Since $a \notin c'(h_F^{-1}(N_{G}(u)))$ we can extend $c$ by letting $c(u)=a$ to obtain a desired proper $k$-coloring of $G$.
    
    This, in particular, implies that if there exists a proper $k$-colouring of $G'$, there also exists a proper $k$-coloring of $G$, so $\chi (G') \ge \chi (G)$. On the other hand, since there is a homomorphism from $G'$ to $G$, we have that $\chi(G')\le \chi(G)$, so $\chi(G') = \chi(G)$.
\end{proof}

\textbf{Remark.} In private communication, Alon suggested an alternative way to prove \Cref{lem:blowup} explicitly. The basic idea is to start with a high-girth Ramanujan graph (such as an LPS-graph \cite{ramanujan}) on $|G|m$ vertices. One can then assign to each vertex of $G$ a disjoint set of $m$ of its vertices and remove any edges within parts or between parts corresponding to non-adjacent vertices of $G$. One can then show, using a similar argument as we did with random subsampling, that the chromatic number of this graph will remain the same as that of $G$, provided $m$ is chosen to be sufficiently large.

By combining \Cref{thm:exponential-chi-orthogonality} with \Cref{thm-double-dimension} and \Cref{lem:explicit blowup}, we obtain \Cref{thm:orth} (since $\sqrt{1.1547} =1.074569...$).
As discussed at the end of \Cref{sec:self}, \Cref{thm:orth} then implies \Cref{main} and \ref{thm:diam} (where $1.074569^{\sqrt{2}}=1.107...$).

\section{Concluding remarks}\label{sec:con}
In this paper, we proved that there exist unit distance and orthogonality graphs in $\R^d$ and diameter graphs in $\R^{d^2}$ with chromatic number exponential in $d$ and with arbitrarily large girth.

We note that there is another natural notion of unit distance graphs  (and similarly for diameter and orthogonality graphs), sometimes referred to as faithful or induced unit distance graphs in which one insists that any pair of points in $\R^d$ at unit distance do make an edge. Unfortunately, due to the fact we subsample edges, our arguments do not apply in this setting and it remains a very interesting open problem to find faithful unit distance graphs with large girth and chromatic number. We point the reader to a paper of Alon and Kupavskii \cite{alon-kupavskii} collecting the differences between the two notions.

We note also that our \Cref{thm-double-dimension} allows us to embed (faithfully) in $\mathbb{S}^{2d-1}$ a blow-up of any (faithful) orthogonality graph in $\mathbb{S}^{d-1}$. This is in general tight since one can embed a complete graph $K_d$ as an orthogonality graph in $\mathbb{S}^{d-1}$ but to embed a complete $d$-partite graph with three vertices in every part onto a sphere one needs at least $2d$ dimensions. Indeed, each triple of vertices making a part defines a plane which should be orthogonal to all other such planes. We thank an anonymous referee for pointing this out.

Another interesting direction which arises is the fact that not all unit distance graphs are orthogonality graphs (of the same dimension), so it might be interesting to determine whether it is possible to prove some analogue of \Cref{thm-double-dimension} for unit distance graphs. If so then one could use the slightly better constructions known for unit distance graphs compared to orthogonality ones and obtain a very slight improvement over our \Cref{main}. Understanding both this and the previous question might help in improving the bounds on the Borsuk Problem without any girth conditions.

\section*{Acknowledgements}

We thank Noga Alon for sharing his alternative explicit proof of \Cref{lem:blowup} and thank Jim Geelen for helpful discussions on topics related to \Cref{lem:blowup}. We also thank Andrei M. Raigorodskii for pointing out that the orthogonality graphs he constructs in \cite{borsuk-best} come with a slightly worse base constant than we initially stated. 
We are grateful to the anonymous referees for their careful reading of the paper and many useful suggestions.

\providecommand{\MR}[1]{}
\providecommand{\MRhref}[2]{%
  \href{http://www.ams.org/mathscinet-getitem?mr=#1}{#2}
}


\begin{thebibliography}{10}

\bibitem{hypergraphs-high-chromatic-and-girth1}
N.~Alon, A.~Kostochka, B.~Reiniger, D.~B. West, and X.~Zhu, \emph{Coloring, sparseness and girth}, Israel J. Math. \textbf{214} (2016), no.~1, 315--331. \MR{3540616}

\bibitem{alon-kupavskii}
N.~Alon and A.~Kupavskii, \emph{Two notions of unit distance graphs}, J. Combin. Theory Ser. A \textbf{125} (2014), 1--17. \MR{3207464}

\bibitem{borsuk}
K.~Borsuk, \emph{Drei s{\"a}tze {\"u}ber die n-dimensionale euklidische sph{\"a}re}, Fundam. Math. \textbf{20} (1933), no.~1, 177--190.

\bibitem{bourgain-lindenstrauss}
J.~Bourgain and J.~Lindenstrauss, \emph{On covering a set in {${\bf R}^N$} by balls of the same diameter}, Geometric aspects of functional analysis (1989--90), Lecture Notes in Math., vol. 1469, Springer, Berlin, 1991, pp.~138--144. \MR{1122618}

\bibitem{Davies2021Box}
J.~Davies, \emph{Box and segment intersection graphs with large girth and chromatic number}, Adv. Comb. (2021), Paper No. 7, 9. \MR{4292574}

\bibitem{davies2021solution}
J.~Davies, C.~Keller, L.~Kleist, S.~Smorodinsky, and B.~Walczak, \emph{A solution to {R}ingel's circle problem}, to appear in J. Eur. Math. Soc (2021).

\bibitem{de-Bruijn-Erdos}
N.~G. de~Bruijn and P.~Erd\"{o}s, \emph{A colour problem for infinite graphs and a problem in the theory of relations}, Indag. Math. \textbf{13} (1951), 369--373, Nederl. Akad. Wetensch. Proc. Ser. A {\bf 54}. \MR{46630}

\bibitem{demekhin2012distance}
E.~Demekhin, A.~Ra\u{\i}gorodski\u{\i}, and O.~Rubanov, \emph{Distance graphs having large chromatic numbers and not containing cliques or cycles of given size}, Sb. Math. \textbf{204} (2013), 508--538.

\bibitem{descartes1947}
B.~Descartes, \emph{A three colour problem}, Eureka \textbf{9} (1947), no.~21, 24--25.

\bibitem{descartes1954solution}
B.~Descartes, \emph{Solution to advanced problem no. 4526}, Amer. Math. Monthly \textbf{61} (1954), no.~352, 216.

\bibitem{hypergraphs-high-chromatic-and-girth2}
D.~Duffus, V.~R\"{o}dl, B.~Sands, and N.~Sauer, \emph{Chromatic numbers and homomorphisms of large girth hypergraphs}, Topics in discrete mathematics, Algorithms Combin., vol.~26, Springer, Berlin, 2006, pp.~455--471. \MR{2249281}

\bibitem{dusart}
P.~Dusart, \emph{The {$k$}-th prime is greater than {$k(\ln k+\ln\ln k-1)$} for {$k\geq2$}}, Math. Comp. \textbf{68} (1999), no.~225, 411--415. \MR{1620223}

\bibitem{erdos59-high-girth-high-chromatic-graph}
P.~Erd\H{o}s, \emph{Graph theory and probability}, Canadian J. Math. \textbf{11} (1959), 34--38. \MR{102081}

\bibitem{erdos-hajnal66-high-girth-high-chromatic-hypergraph}
P.~Erd\H{o}s and A.~Hajnal, \emph{On chromatic number of graphs and set-systems}, Acta Math. Acad. Sci. Hungar. \textbf{17} (1966), 61--99. \MR{193025}

\bibitem{local-lemma-paper}
P.~Erd\H{o}s and L.~Lov\'{a}sz, \emph{Problems and results on {$3$}-chromatic hypergraphs and some related questions}, Infinite and finite sets ({C}olloq., {K}eszthely, 1973; dedicated to {P}. {E}rd\H{o}s on his 60th birthday), {V}ols. {I}, {II}, {III}, Colloq. Math. Soc. J\'{a}nos Bolyai, vol. Vol. 10, North-Holland, Amsterdam-London, 1975, pp.~609--627. \MR{382050}

\bibitem{erdos1975unsolved}
P.~Erd{\H{o}}s, \emph{Unsolved problems}, Congress Numerantium XV--Proceedings of the 5th British Comb. Conf., 1975, p.~681.

\bibitem{frankl-wilson}
P.~Frankl and R.~M. Wilson, \emph{Intersection theorems with geometric consequences}, Combinatorica \textbf{1} (1981), no.~4, 357--368. \MR{647986}

\bibitem{graham}
R.~L. Graham, \emph{Euclidean ramsey theory}, Handbook of Discrete and Computational Geometry, Chapman and Hall/CRC, 2017, pp.~281--297.

\bibitem{hadwiger-nelson}
H.~Hadwiger, \emph{Ein {U}eberdeckungssatz f\"{u}r den {E}uklidischen {R}aum}, Portugal. Math. \textbf{4} (1944), 140--144. \MR{11108}

\bibitem{hochberg}
R.~Hochberg and P.~O'Donnell, \emph{Some {$4$}-chromatic unit-distance graphs without small cycles}, Geombinatorics \textbf{5} (1996), no.~4, 137--141. \MR{1380143}

\bibitem{kahn-kalai}
J.~Kahn and G.~Kalai, \emph{A counterexample to {B}orsuk's conjecture}, Bull. Amer. Math. Soc. (N.S.) \textbf{29} (1993), no.~1, 60--62. \MR{1193538}

\bibitem{hypergraphs-high-chromatic-and-girth3}
A.~V. Kostochka and J.~Ne\v{s}et\v{r}il, \emph{Properties of {D}escartes' construction of triangle-free graphs with high chromatic number}, Combin. Probab. Comput. \textbf{8} (1999), no.~5, 467--472. \MR{1731981}

\bibitem{borsuk-alternative}
A.~Kupavskii and A.~Raigorodskii, \emph{Counterexamples to {B}orsuk's conjecture on spheres of small radii}, Mosc. J. Comb. Number Theory \textbf{2} (2012), no.~4, 27--48. \MR{3065279}

\bibitem{kupavskii2012distance}
A.~B. Kupavskii, \emph{Distance graphs with large chromatic number and arbitrary girth}, Mosc. J. Comb. Number Theory \textbf{2} (2012), no.~2, 52--62. \MR{2988526}

\bibitem{schur}
A.~B. Kupavskii and A.~Polyanskii, \emph{Proof of {S}chur's conjecture in {$\Bbb R^D$}}, Combinatorica \textbf{37} (2017), no.~6, 1181--1205. \MR{3759913}

\bibitem{larman-rogers}
D.~G. Larman and C.~A. Rogers, \emph{The realization of distances within sets in {E}uclidean space}, Mathematika \textbf{19} (1972), 1--24. \MR{319055}

\bibitem{lovasz-explicit}
L.~Lov\'{a}sz, \emph{On chromatic number of finite set-systems}, Acta Math. Acad. Sci. Hungar. \textbf{19} (1968), 59--67. \MR{220621}

\bibitem{ramanujan}
A.~Lubotzky, R.~Phillips, and P.~Sarnak, \emph{Ramanujan graphs}, Combinatorica \textbf{8} (1988), no.~3, 261--277. \MR{963118}

\bibitem{mohar2023subgraphs}
B.~Mohar and H.~Wu, \emph{Subgraphs of {K}neser graphs with large girth and large chromatic number}, Art Discrete Appl. Math. \textbf{6} (2023), no.~2, Paper No. 2.11, 7. \MR{4556991}

\bibitem{nevsetvril1979short}
J.~Ne\u~set\u ril and V.~e. R\"odl, \emph{A short proof of the existence of highly chromatic hypergraphs without short cycles}, J. Combin. Theory Ser. B \textbf{27} (1979), no.~2, 225--227. \MR{546865}

\bibitem{nesetril-survey}
J.~Ne\u{s}et\u{r}il, \emph{A combinatorial classic---sparse graphs with high chromatic number}, Erd\"{o}s centennial, Bolyai Soc. Math. Stud., vol.~25, J\'{a}nos Bolyai Math. Soc., Budapest, 2013, pp.~383--407. \MR{3203606}

\bibitem{hypergraphs-high-chromatic-and-girth6}
J.~Ne\u{s}et\u{r}il and V.~R\"{o}dl, \emph{On a probabilistic graph-theoretical method}, Proc. Amer. Math. Soc. \textbf{72} (1978), no.~2, 417--421. \MR{507350}

\bibitem{nilli}
A.~Nilli, \emph{On {B}orsuk's problem}, Jerusalem combinatorics '93, Contemp. Math., vol. 178, Amer. Math. Soc., Providence, RI, 1994, pp.~209--210. \MR{1310585}

\bibitem{odonell}
P.~O'Donnell, \emph{Arbitrary girth, 4-chromatic unit distance graphs in the plane. {I}. {G}raph description}, Geombinatorics \textbf{9} (2000), no.~3, 145--152. \MR{1746081}

\bibitem{odonnell2}
P.~O'Donnell, \emph{Arbitrary girth, 4-chromatic unit distance graphs in the plane. {II}. {G}raph embedding}, Geombinatorics \textbf{9} (2000), no.~4, 180--193. \MR{1763978}

\bibitem{PACH202429}
J.~Pach, G.~Tardos, and G.~T\'oth, \emph{Disjointness graphs of short polygonal chains}, J. Combin. Theory Ser. B \textbf{164} (2024), 29--43. \MR{4642963}

\bibitem{PV90}
H.~J. Pr\"omel and B.~Voigt, \emph{A sparse {G}allai-{W}itt theorem}, Topics in combinatorics and graph theory ({O}berwolfach, 1990), Physica, Heidelberg, 1990, pp.~747--755. \MR{1100099}

\bibitem{promel1988sparse}
H.~J. Pr\"omel and B.~Voigt, \emph{A sparse {G}raham-{R}othschild theorem}, Trans. Amer. Math. Soc. \textbf{309} (1988), no.~1, 113--137. \MR{957064}

\bibitem{prosanov2019counterexamples}
R.~I. Prosanov, \emph{Counterexamples to {B}orsuk's conjecture that have large girth}, Mat. Zametki \textbf{105} (2019), no.~6, 890--898. \MR{3954318}

\bibitem{prosanov-hadwiger-nelson}
R.~Prosanov, \emph{A new proof of the {L}arman-{R}ogers upper bound for the chromatic number of the {E}uclidean space}, Discrete Appl. Math. \textbf{276} (2020), 115--120. \MR{4075526}

\bibitem{raigorodskii-bound}
A.~M. Ra\u{\i}gorodski\u{\i}, \emph{On the chromatic number of a space}, Uspekhi Mat. Nauk \textbf{55} (2000), no.~2(332), 147--148. \MR{1781075}

\bibitem{hadwiger-nelson-survey}
A.~M. Ra\u{\i}gorodski\u{\i}, \emph{The {E}rd{\H{o}}s-{H}adwiger problem and the chromatic numbers of finite geometric graphs}, Mat. Sb. \textbf{196} (2005), no.~1, 123--156. \MR{2141326}

\bibitem{raigorodskii07}
A.~M. Ra\u{\i}gorodski\u{\i}, \emph{On distance graphs with large chromatic number but without large simplices}, Russ. Math. Surv. \textbf{62} (2007), no.~6, 1224.

\bibitem{raigorodski-survey}
A.~M. Ra\u{\i}gorodski\u{\i}, \emph{Coloring distance graphs and graphs of diameters}, Thirty essays on geometric graph theory, Springer, New York, 2013, pp.~429--460. \MR{3205167}

\bibitem{raigorodskii-shabanov-survey}
A.~M. Ra\u{\i}gorodski\u{\i} and D.~A. Shabanov, \emph{The {E}rd{\H{o}}s-{H}ajnal problem of hypergraph colouring, its generalizations, and related problems}, Russ. Math. Surv. \textbf{66} (2011), no.~5, 933.

\bibitem{borsuk-best}
A.~M. Ra\u{\i}gorodski\u{\i}, \emph{On a bound in {B}orsuk's problem}, Uspekhi Mat. Nauk \textbf{54} (1999), no.~2, 185--186.

\bibitem{borsuk-survey}
A.~M. Ra\u{\i}gorodski\u{\i}, \emph{Three lectures on the {B}orsuk partition problem}, Surveys in contemporary mathematics, London Math. Soc. Lecture Note Ser., vol. 347, Cambridge Univ. Press, Cambridge, 2008, pp.~202--247. \MR{2388494}

\bibitem{rubanov}
O.~I. Rubanov, \emph{Chromatic numbers of three-dimensional distance graphs without tetrahedra}, Mat. Zametki \textbf{82} (2007), no.~5, 797--800. \MR{2399959}

\bibitem{sagdeev2017lower}
A.~A. Sagdeev, \emph{On lower bounds for the chromatic numbers of distance graphs with large girth}, Mat. Zametki \textbf{101} (2017), no.~3, 430--445. \MR{3635434}

\bibitem{sagdeev2018improved}
A.~A. Sagdeev, \emph{An improved {F}rankl-{R}\"odl theorem and some of its geometric consequences}, Problemy Peredachi Informatsii \textbf{54} (2018), no.~2, 45--72. \MR{3845496}

\bibitem{sagdeev2019frankl}
A.~A. Sagdeev and A.~M. Raigorodskii, \emph{On a {F}rankl-{W}ilson theorem and its geometric corollaries}, Acta Math. Univ. Comenian. (N.S.) \textbf{88} (2019), no.~3, 1029--1033. \MR{4012909}

\bibitem{schramm}
O.~Schramm, \emph{Illuminating sets of constant width}, Mathematika \textbf{35} (1988), no.~2, 180--189. \MR{986627}

\bibitem{wormald}
N.~Wormald, \emph{A {$4$}-chromatic graph with a special plane drawing}, J. Austral. Math. Soc. Ser. A \textbf{28} (1979), no.~1, 1--8. \MR{541161}

\bibitem{zhu1996uniquely}
X.~Zhu, \emph{Uniquely {$H$}-colorable graphs with large girth}, J. Graph Theory \textbf{23} (1996), no.~1, 33--41. \MR{1402136}

\end{thebibliography}

\providecommand{\bysame}{\leavevmode\hbox to3em{\hrulefill}\thinspace}
\providecommand{\MR}{\relax\ifhmode\unskip\space\fi MR }
\providecommand{\MRhref}[2]{%
  \href{http://www.ams.org/mathscinet-getitem?mr=#1}{#2}
}
\providecommand{\href}[2]{#2}

\end{document}